\documentclass[12pt]{amsart}
\usepackage{amsmath}
\usepackage{amsxtra}
\usepackage{amstext}
\usepackage{amssymb}
\usepackage{amsthm}
\usepackage{latexsym}
\usepackage{dsfont} 
\usepackage{verbatim}
\usepackage{tabls}
\usepackage{graphicx}
\usepackage{rotating}
\usepackage{caption}
\theoremstyle{plain}
\newtheorem{thm}{Theorem}[section]

\newtheorem{lem}[thm]{Lemma}
\newtheorem{prop}[thm]{Proposition}

\theoremstyle{definition}

\newtheorem{cla}[thm]{Claim}

\numberwithin{equation}{section}
\def\supp{\operatorname{supp}}

\def\eps{\varepsilon}

\def\dist{\operatorname{dist}}

\def\XXint#1#2#3{{\setbox0=\hbox{$#1{#2#3}{\int}$}
     \vcenter{\hbox{$#2#3$}}\kern-.5\wd0}}

\begin{document}

\title[Three revolutions]
{Three revolutions in the kernel are worse than one}

\author[B. Jaye]
{Benjamin Jaye}
\address{Department of Mathematical Sciences,
Kent State University,
Kent, OH 44240, USA}
\email{bjaye@kent.edu}

\author[F. Nazarov]
{Fedor Nazarov}
\email{nazarov@math.kent.edu}

\date{\today}

\begin{abstract}An example is constructed of a purely unrectifiable measure $\mu$ for which the singular integral associated to the kernel $K(z) = \frac{\bar{z}}{z^2}$ is bounded in $L^2(\mu)$.  The singular integral fails to exist in the sense of principal value $\mu$-almost everywhere.  This is in sharp contrast with the results known for the kernel $\tfrac{1}{z}$ (the Cauchy transform).\end{abstract}

\maketitle

\section{Introduction}

Let $B(z,r)$ denote the closed disc in $\mathbb{C}$ centred at $z$ with radius $r>0$.  A finite Borel measure $\mu$ is said to be $1$-dimensional if $\mathcal{H}^1(\supp(\mu))<\infty$, and there exists a constant $C>0$ such that $\mu(B(z,r))\leq Cr$ for any $z\in \mathbb{C}$ and $r>0$.

For a kernel function $K:\mathbb{C}\backslash \{0\}\rightarrow \mathbb{C}$, and a finite measure $\mu$, we define the singular integral operator associated to $K$ by
$$T_{\mu}(f)(z) = \int_{\mathbb{C}}K(z-\xi) f(\xi) d\mu(\xi), \text{ for }z\not\in \supp(\mu).$$

A well-known problem in harmonic analysis is to determine geometric properties of $\mu$ from regularity properties of the operator $T_{\mu}$, see for instance the monograph of David and Semmes \cite{DS}.  This paper concerns the question of characterizing those functions $K$ with the following property:
\begin{equation}\tag{$*$}\begin{split}
   &\textit{Let }\mu \textit{ be a 1-dimensional measure. Then } \\
   &\|T_{\mu}(1)\|_{L^{\infty}(\mathbb{C}\backslash \supp(\mu))}<\infty\textit{ implies that }\mu \textit{ is rectifiable}.
\end{split}\end{equation}


The property that $\|T_{\mu}(1)\|_{L^{\infty}(\mathbb{C}\backslash \supp(\mu))}<\infty$ is equivalent to the boundedness of $T_{\mu}$ as an operator in $L^2(\mu)$, see for instance \cite{NTV}.  A measure $\mu$ is rectifiable if $\supp(\mu)$ can be covered (up to an exceptional set of $\mathcal{H}^1$ measure zero) by a countable union of rectifiable curves.  A measure $\mu$ is purely unrectifiable if its support is purely unrectifiable, that is, $\mathcal{H}^1(\Gamma\cap \supp(\mu))=0$ for any rectifiable curve $\Gamma$.

David and L\'{e}ger  \cite{Leg} proved that the Cauchy kernel $\tfrac{1}{z}$ has property ($*$).  As is remarked in \cite{CMPT}, the proof in \cite{Leg} extends to the case when the Cauchy kernel is replaced by either its real or imaginary part, i.e. $\tfrac{\Re(z)}{|z|^2}$ or $\tfrac{\Im(z)}{|z|^2}$.  Recently in \cite{CMPT},  Chousionis,  Mateu, Prat, and  Tolsa extended the result of \cite{Leg} and showed that kernels of the form $\tfrac{(\Re (z))^k}{|z|^{k+1}}$ have property ($*$) for any odd positive integer $k$.  Both of these results use the Melnikov-Menger curvature method.  

On the other hand, Huovinen \cite{Huo2} has shown that there is a purely unrectifiable Ahlfors-David (AD)-regular set $E$ for which the singular integral associcated to the kernel $\tfrac{\Re(z)}{|z|^2}-\tfrac{\Re(z)^3}{|z|^4}$ is bounded in $L^2(\mathcal{H}^1_{|E})$.  In fact, an essentially stronger conclusion is proved that the principal values of the associated singular integral operator exist $\mathcal{H}^1$-a.e. on $E$.  Huovinen takes advantage of several non-standard symmetries and cancellation properties in this kernel to construct his very nice example.

The result of this paper is that a weakened version of Huovinen's theorem holds for a very simple kernel function.  Indeed, it is perhaps the simplest example of a kernel for which the Menger curvature method fails to be directly applicable.  From now on, we shall fix \begin{equation}\label{kernel}K(z) = \frac{\bar{z}}{z^2}, z\in \mathbb{C}\backslash \{0\}.\end{equation}  
We prove the following result.

\begin{thm}\label{thm1}  There exists a $1$-dimensional purely unrectifiable probability measure $\mu$ with the property that $\|T_{\mu}(1)\|_{L^{\infty}(\mathbb{C}\backslash \supp(\mu))}<\infty.$
\end{thm}

In other words, the kernel $K$ in (\ref{kernel}) fails to satisfy property ($*$).  At this point, we would also like to mention Huovinen's thesis work \cite{Huo1}, regarding the kernel function $K(z)$ from (\ref{kernel}).  It is proved that if $\liminf_{r\rightarrow0}\frac{\mu(B(z,r))}{r}\in (0,\infty)$ $\mu$-a.e. (essentially the AD-regularity of $\mu$), then the $\mu$-almost everywhere existence of $T_{\mu}(1)$ in the sense of principal value implies that $\mu$ is rectifiable.  This result was proved by building upon the theory of symmetric measures, developed by Mattila \cite{Mat95b}, and Mattila and Preiss \cite{MP95}. Unfortunately the measure in Theorem \ref{thm1} does not satisfy the AD-regularity condition.  In view of Huovinen's work it would be of interest to construct an AD-regular measure supported on an unrectifiable set for which the conclusion of Theorem \ref{thm1} holds.  We have not been able to construct such a measure (yet).

For the measure $\mu$ constructed in Theorem \ref{thm1}, we show that $T_{\mu}(1)$ fails to exist in the sense of principal value $\mu$-almost everywhere.  Thus the two properties of $L^2(\mu)$ boundedness of the operator $T_{\mu}$, and the existence of $T_{\mu}(1)$ in the sense of principal value, are quite distinct for this singular integral operator. 


\section{Notation}
\begin{itemize}
\item Let $m_2$ denote the $2$-dimensional Lebesgue measure normalized so that $m_2(B(0,1))=1$.  We let $m_1$ denote the $1$-dimensional Lebesgue measure.
\item A collection of squares are essentially pairwise disjoint if the interiors of any two squares in the collection do not intersect.  Throughout the paper, all squares are closed.
\item We shall denote by $C$ and $c$ large and small absolute positive constants.  The constant $C$ should be thought of as large (at least $1$), while $c$ is to be thought of as small (smaller than $1$).
\item For $a>1$, the disc $aB$ denotes the concentric enlargement of a disc $B$ by a factor of $a$.
\item We define the $\mathcal{H}^1$-measure of a set $E$ by\\ $\mathcal{H}^1(E) = \sup_{\delta>0}\inf\bigl\{\sum_j r_j \,:\, E\subset \bigcup_j B(x_j,r_j) \text{ with }r_j\leq \delta \bigl\}.$
\item For $z\in \mathbb{C}$ and $r>0$,  we define the annulus $A(z,r) = B(z,r)\backslash B(z, \tfrac{r}{2})$.
\item The set $\supp(\mu)$ denotes the closed support of $\mu$.
\end{itemize}
\section{A reflectionless measure}

Let us make the key observation that allows us to prove Theorem \ref{thm1}.

\begin{lem}\label{refl}  Let $z\in \mathbb{C}$, $r>0$.  For any $\omega \in B(z,r)$,
$$\int_{B(z,r)}K(\omega-\xi)dm_2(\xi)=0.
$$
\end{lem}

\begin{proof}
Without loss of generality, we may set $z=0$ and $r=1$.  If $|\omega|<|\xi|$, then $$ K(\omega-\xi) = \frac{\overline{\omega - \xi}}{\xi^2}\sum_{\ell=0}^{\infty}(\ell+1)\Bigl(\frac{\omega}{\xi}\Bigl)^{\ell}.$$
So whenever $t>|\omega|$, we have $ \int_{\partial B(0,t)} K(\omega-\xi) dm_1(\xi)=0$.  (This follows merely from the fact that $\int_{\partial B(0,t)}\bar{\xi}^{\ell}\xi^k dm_1(\xi)=0$ whenever $k, \ell\in \mathbb{Z}$ satisfy $k\neq \ell$.)  On the other hand, if $|\xi|<|\omega|$, then
$$K(\omega-\xi) = \frac{\overline{\omega - \xi}}{\omega^2}\sum_{\ell=0}^{\infty}(\ell+1)\Bigl(\frac{\xi}{\omega}\Bigl)^{\ell}.
$$
Therefore, if $t<|\omega|$, then
$$\int_{\partial B(0,t)} K(\omega-\xi) dm_1(\xi) = 2\pi\Bigl[ t\frac{\overline{\omega}}{\omega^2}-2 \frac{t^3}{\omega^3} \Bigl]= \frac{2\pi}{\omega^3}(t|\omega|^2-2t^3).
$$
Since $\int_0^{|\omega|}(t|\omega|^2-2t^3) dt=0$, the desired conclusion follows.
\end{proof}


The next lemma will form the basis of the proof of the non-existence of $T_{\mu}(1)$ in the sense of principal value.

\begin{lem}\label{lebmisbal} There exists a constant $\tilde{c}>0$ such that for any disc $B(z,r)$, and $\omega \in\partial B(z,r)$,
$$\Bigl|\int_{A(\omega, r)\cap B(z,r)} K(\omega-\xi) \frac{dm_2(\xi)}{r}\Bigl|\geq \tilde{c}.
$$
\end{lem}

\begin{figure}[t]\label{balance}
\centering
 \includegraphics[width = 110mm]{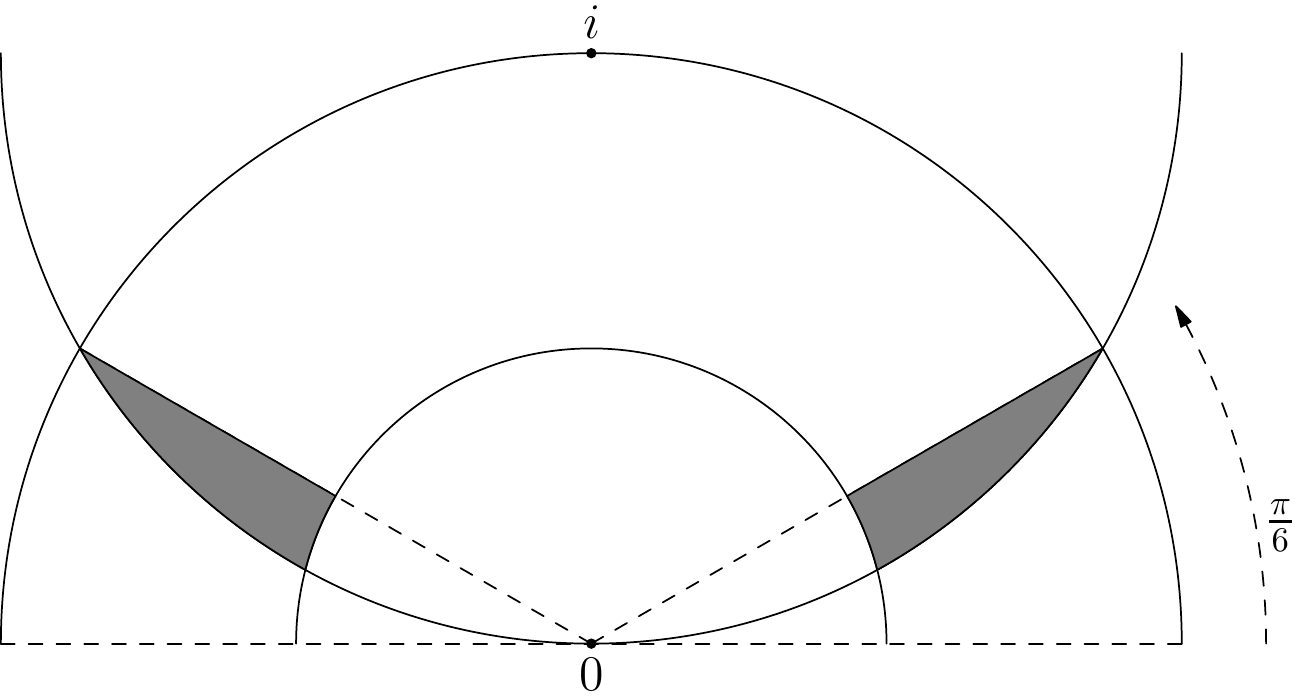}
\caption[Blah.]{The set-up for the proof of Lemma \ref{lebmisbal}.}
 \end{figure}

\begin{proof}By an appropriate translation and rescaling, we may assume that $B(z,r) = B(i,1)$, and $\omega=0$.  Making reference to Figure 1 above, we split the domain of integration into three regions, $I = \{\xi \in A(0,1): \arg(\xi)\in \bigl[\frac{\pi}{6}, \frac{5\pi}{6}]\}$, $II = \{\xi \in A(0,1)\cap B(i,1): \arg(\xi)\in \bigl[0,\frac{\pi}{6}\bigl]\}$ and $III = \{\xi \in A(0,1)\cap B(i,1): \arg(\xi)\in \bigl[\frac{5\pi}{6}, \pi\bigl]\}$.   The regions $II$ and $III$ are respectively the right and left grey shaded regions in Figure 1.  Note that $\Im K(-\xi) <0$ if $\arg(\xi)\in \bigl[\frac{\pi}{3}, \frac{2\pi}{3}\bigl]$, and $\Im K(-\xi) >0$ if $\arg(\xi)\in \bigl[0,\frac{\pi}{3}\bigl]\cup \bigl[\frac{2\pi}{3}, \pi \bigl].$  Furthermore, note that $$\int_{I} \Im K(-\xi) dm_2(\xi) = \frac{1}{\pi}\int_{\tfrac{1}{2}}^1\frac{1}{t}\int_{\tfrac{\pi}{6}}^{\tfrac{5\pi}{6}}-\Im \bigl(e^{-3\theta i}\bigl)td\theta dt =0.$$
But $\int_{II\cup III} \Im K(-\xi) dm_2(\xi) =  2 \int_{II} \Im K(-\xi) dm_2(\xi)>0$.  Therefore, by setting $\tilde{c} = 2 \int_{II} \Im K(-\xi) dm_2(\xi)$, the lemma follows.\end{proof}

\section{Packing squares in a disc}

Fix $r,R\in (0,\infty)$ such that $r<\tfrac{R}{16}$ and $\tfrac{R}{r}\in \mathbb{N}$.

\begin{lem}\label{squarepac}  One can pack $\tfrac{R}{r}$ pairwise essentially disjoint squares of side length $\sqrt{\pi r R}$ into a disc of radius $R(1+4\sqrt{\tfrac{r}{R}})$.\end{lem}

\begin{proof}We may assume that the disc is centred at the origin.  Consider the square lattice with mesh size $\sqrt{\pi r R}$.  Label those squares that intersect $B(0,R)$ as $Q_1,\dots, Q_M$.  These squares are contained in $B(0, R(1+4\sqrt{\tfrac{r}{R}}))$.  Since $M rR = \sum_{j=1}^M m_2(Q_j)> m_2(B(0,R)) = R^2$, we have that $M> \tfrac{R}{r}$.  By throwing away $M-\tfrac{R}{r}$ of the least desirable squares, we arrive at the desired collection.\end{proof}

\begin{lem}\label{smalldiff} Consider a disc $B(z,R)$.  Let $Q_1, \dots, Q_{R/r}$ be the collection of squares contained in $B(z, R(1+4\sqrt{\tfrac{r}{R}}))$ found in Lemma \ref{squarepac}.  Then $m_2(B(z,R)\triangle \bigcup_{j=1}^{R/r} Q_j)\leq C r^{1/2} R^{3/2}.$


\end{lem}

\begin{proof}
Since $m_2(B(z,R)) = m_2\bigl(\bigcup_{j=1}^{R/r} Q_j)=R^2$, the property follows from the fact that both sets are contained in $B(z, R(1+4\sqrt{\tfrac{r}{R}}))$.  
\end{proof}

\section{The construction of the sparse Cantor set $E$}

Let $r_0=1$, and choose $r_j$, $j\in \mathbb{N}$, to be a sequence which tends to zero quickly.  Assume that $r_j<\tfrac{r_{j-1}}{100}$, $\tfrac{1}{r_j}\in \mathbb{N}$, and $\tfrac{r_j}{r_{j+1}}\in \mathbb{N}$ for all $j\geq 1$.

Several additional requirements will be imposed on the decay of $r_j$  over the course of the following analysis, and we make no attempt to optimize the conditions. 

It will be convenient to let $s_{n+1} = 4\sqrt{\tfrac{r_{n+1}}{r_{n}}}$ for $n\in \mathbb{Z}_+$.

First define $\widetilde{B}^{(0)}_1=B(0,1)$.  Given the $n$-th level collection of $\tfrac{1}{r_n}$ discs $\widetilde{B}^{(n)}_j$ of radius $r_n$, we construct the $(n+1)$-st generation according to the following procedure:

Fix a disc $\widetilde{B}^{(n)}_j$.  Apply Lemma \ref{squarepac} with $R=r_n$ and $r=r_{n+1}$ to find $\tfrac{r_n}{ r_{n+1}}$ squares $Q_{\ell}^{(n+1)}$ of side length $\sqrt{\pi r_{n+1} r_n}$ that are pairwise essentially disjoint, and contained in $(1+s_{n+1})\cdot \widetilde{B}^{(n)}_j$.   Let $z_{\ell}^{(n+1)}$ be the centre of $Q_{\ell}^{(n+1)}$, and set $\widetilde{B}_{\ell}^{(n+1)} = B(z_{\ell}^{(n+1)}, r_{n+1})$.    This procedure is carried out for each disc $\widetilde{B}^{(n)}_j$ from the $n$-th level collection.  There are a total of $\tfrac{1}{r_{n+1}}$ discs $\widetilde{B}^{(n+1)}_{\ell}$ in the $(n+1)$-st level.

The above construction is executed for each $n\in \mathbb{Z}_{+}$.

Now, set $B^{(n)}_j = (1+s_{n+1})\widetilde B^{(n)}_j$.  Define $E^{(n)}= \bigcup_j B^{(n)}_j$.  We shall repeatedly use the following properties of the construction:

(a) $\bigcup_{\ell} Q^{(n+1)}_{\ell}\subset E^{(n)}$, for all $n\geq 0$.

(b)  $B^{(n)}_j\subset Q^{(n)}_j$ for each $n\geq 1$. Moreover, $\text{dist}(B^{(n)}_j, \partial Q^{(n)}_j)\geq \tfrac{1}{2}\sqrt{r_{n-1}r_{n}}$.

(c)  $\text{dist}(B^{(n)}_j, B^{(n)}_k)\geq \tfrac{1}{2}\sqrt{r_{n-1}r_{n}}$ whenever $j\neq k$, $n\geq 0$.

\begin{figure}[t]\label{discpic}
\centering
 \includegraphics[width = 100mm]{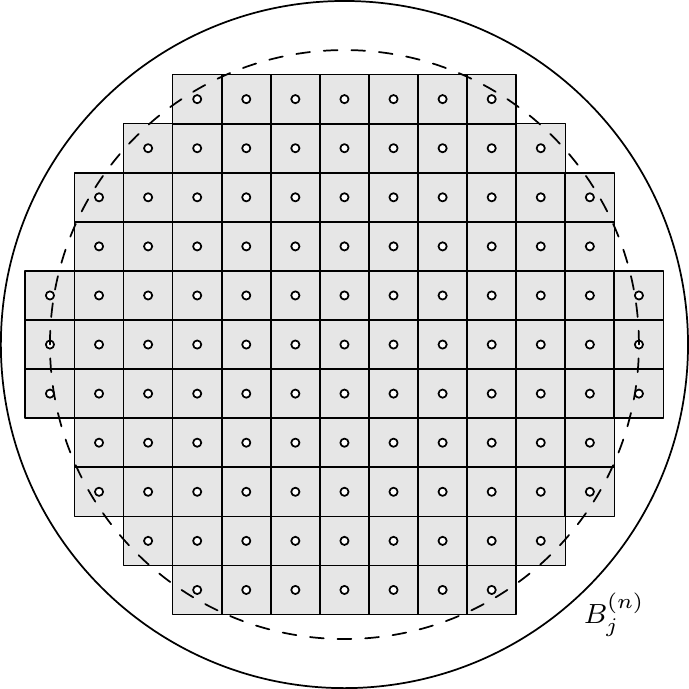}
\caption[Blah.]{The picture shows a single disc $B^{(n)}_j$ of radius $(1+s_{n+1})r_n$.  The grey shaded squares are the squares $Q^{(n+1)}_{\ell}$ of sidelength $\sqrt{\pi r_{n} r_{n+1}}$ formed by applying Lemma \ref{squarepac} to the disc $\widetilde{B}^{(n)}_j$ of radius $r_n$.  The boundary of the disc $\widetilde{B}^{(n)}_j$ is the dashed circle.  Deep inside each square $Q^{(n+1)}_{\ell}$ is the disc $B^{(n+1)}_{\ell}$ of radius $(1+s_{n+2})r_{n+1}$.}
 \end{figure}

Property (a) is immediate.  To see property (b), merely note that $\text{dist}(B^{(n)}_j, \partial Q^{(n)}_j)=\tfrac{\sqrt{\pi r_{n-1}r_n}}{2}- (1+s_{n+1})r_n\geq \tfrac{1}{2}\sqrt{r_{n-1}r_n}.$  For property (c), we shall use induction. If $n=0$, then the claim is trivial.  Using (b), the claimed estimate is clear if $Q^{(n)}_j$ and $Q^{(n)}_k$ have been created by an application of Lemma \ref{squarepac} in a common disc $\widetilde{B}^{(n-1)}_{\ell}$.  Otherwise, the squares are born out of applying Lemma \ref{squarepac} to different discs at the $(n-1)$-st level, and those parent discs are already separated by $\tfrac{1}{2}\sqrt{r_{n-2}r_{n-1}}$.

Courtesy of properties (a) and (b), we see that $E^{(n+1)}\subset E^{(n)}$ for each $n\geq 0$.  Set $E=\bigcap_{n\geq 0}E^{(n)}$.  Each $z\in E^{(n)}$ is contained in a unique disc $B^{(n)}_j$ (or square $Q^{(n)}_j$) which we shall denote by $B^{(n)}(z)$ (respectively $Q^{(n)}(z)$).

If $m\geq n\geq 0$, then $E\cap B^{(n)}_j$ is covered by the $\tfrac{r_n}{r_m}$ discs $B^{(m)}_{\ell}$ that are contained in $B^{(n)}_j$, each of which has radius $(1+s_{m+1})r_m \leq 2r_m$.  Therefore $\mathcal{H}^1(E\cap B^{(n)}_j)\leq 2r_n$.  Taking $n=0$ yields $\mathcal{H}^1(E)\leq 2$.

\section{The measure $\mu$}\label{meas}

Define $\mu_j^{(n)} = \frac{1}{r_n}\chi_{\widetilde{B}^{(n)}_j}m_2$.  Set $\mu^{(n)} = \sum_{j} \mu_j^{(n)}$.  Then $\supp(\mu^{(n)}) \subset E^{(n)}$, and $\mu^{(n)}(\mathbb{C})=1$ for all $n$.  Therefore, there exists a subsequence of the sequence of measures $\mu^{(n)}$ that converges weakly to a measure $\mu$, with $\mu(\mathbb{C})=1$ and $\supp(\mu)\subset E$.

The following three properties hold:

(i)  $\supp(\mu^{(m)})\subset \bigcup_j B_j^{(n)}$ whenever $m\geq n$,

(ii)  $\mu^{(m)}(B^{(n)}_j) = r_n$ for $m\geq n$, and

(iii)  there exists $C_0>0$ such that $\mu^{(n)}(B(z,r))\leq C_0r$ for any $z\in \mathbb{C}$, $r>0$ and $n\geq 0$.

Properties (i) and (ii) follow immediately from the construction of $E^{(n)}$.  To see the third property, note that since $\mu^{(n)}$ is a probability measure, the property is clear if $r\geq 1$.  If $r<1$, then $r\in (r_{m+1}, r_{m})$ for some $m\in \mathbb{Z}_+$.  If $m\geq n$, then $B(z,r)$ intersects at most one disc $B^{(n)}_j$.  Then $\mu^{(n)}(B(z,r)) = \frac{1}{r_n} m_2(B(z,r)\cap \widetilde{B}^{(n)}_j)\leq \frac{r^2}{r_n}\leq r$.  Otherwise $m<n$.  In this case, note that since the discs $B^{(m+1)}_j$ are $\tfrac{1}{2}\sqrt{r_{m}r_{m+1}}$ separated,  $B(z,r)$ intersects at most $1+ C\bigl(\frac{r}{\sqrt{r_mr_{m+1}}}\bigl)^2$ discs $B^{(m+1)}_j$. Hence, by property (ii), we see that
$$\mu^{(n)}(B(z,r)) =\sum_j \mu^{(n)}(B(z,r)\cap B^{(m+1)}_j)\leq \Bigl[1+ C\Bigl(\frac{r}{\sqrt{r_mr_{m+1}}}\Bigl)^2\Bigl] r_{m+1},
$$
which is at most $Cr$.

The weak convergence of a subsequence of $\mu^{(n)}$ to the measure $\mu$, along with property (iii), yields that $\mu(B(z,r))\leq C_0r$ for any disc $B(z,r)$.    We shall henceforth refer to this property by saying that $\mu$ \textit{is $C_0$-nice}.  We have now shown that $\mu$ is $1$-dimensional.

Notice that we also have $\mathcal{H}^1(E)\geq \tfrac{1}{C_0}\mu(E)>0.$

\section{The boundedness of $T_{\mu}(1)$ off the support of $\mu$}


As a simple consequence of the weak convergence of $\mu^{(n)}$ to $\mu$, the property that $\|T_{\mu}(1)\|_{L^{\infty}(\mathbb{C}\backslash \supp(\mu))}<\infty$ will follow from the following proposition.

\begin{prop}\label{maxest} Provided that $\sum_{n\geq 1}\sqrt{s_n}<\infty$, there exists a constant $C>0$ so that the following holds:

Suppose that $\dist(z,\supp(\mu))=\eps>0$.  Then for any $m\in \mathbb{Z}_+$ with $r_m<\tfrac{\eps}{4}$,
$$
\Bigl|\int_{\mathbb{C}} K(z-\xi) d\mu^{(m)}(\xi)\Bigl|\leq C.
$$
\end{prop}

To begin the proof, fix $r_m$ with $r_m<\tfrac{\eps}{4}$.  Let $z^*\in \supp(\mu)$ with $\dist(z,z^*)=\eps$.  For any $\xi\in \supp(\mu)$,  $B^{(m)}(\xi)\cap \supp(\mu^{(m)})\neq \varnothing$, so $\dist(z, \supp(\mu^{(m)}))\geq \eps - (1+s_{m+1}) r_m\geq \tfrac{\eps}{2}$.

Now, let $q$ be the least integer with $r_q\leq \eps$ (so $m\geq q$).  Then by property (ii) of the previous section,
\begin{equation}\label{qest}\int_{B^{(q)}(z^*)} |K(z,\xi)|d\mu^{(m)}(\xi) \leq \frac{2}{\eps}\mu^{(m)}(B^{(q)}(z^*)) = \frac{2r_{q}}{\eps}\leq 2.
\end{equation}

The crux of the matter is the following lemma.

\begin{lem}\label{scale}There exists $C>0$ such that for any $n\in \mathbb{Z}_+$ with $1\leq n\leq q$,
$$\Bigl|\int_{B^{(n-1)}(z^*)\backslash B^{(n)}(z^*)} K(z-\xi) d\mu^{(m)}(\xi)\Bigl| \leq C\sqrt{s_n} + C\sqrt{\tfrac{\eps}{r_{n-1}}}.
$$
\end{lem}

For the proof of Lemma \ref{scale}, we shall require the following simple comparison estimate.

\begin{lem}\label{comparison}  Let $z_0\in \mathbb{C}$, and $\lambda>0$.   Fix $r,R\in (0,1]$ with $100r\leq R$.  Suppose that $\nu_1$ and $\nu_2$ are Borel measures, such that $\supp(\nu_1)\subset Q(z_0, \sqrt{\pi Rr})=Q$, $\supp(\nu_2)\subset B(z_0, 2r)=B$, and $\nu_1(\mathbb{C})=\nu_2(\mathbb{C})$.  Then, for any $z\in \mathbb{C}$ with $\text{dist}(z,Q)\geq \lambda\sqrt{rR}$, we have
\begin{equation}\begin{split}\nonumber\Bigl|\int_Q K(z-\xi)&d\nu_1(\xi)-\int_B K(z-\xi) d\nu_2(\xi)\Bigl|\\
&\leq \frac{1}{\lambda^2}\int_Q \frac{C\sqrt{Rr}}{|z-\xi|^2} d\nu_1(\xi) +\frac{1}{\lambda^2}\int_B \frac{Cr}{|z-\xi|^2} d\nu_2(\xi).
\end{split}\end{equation}
\end{lem}

\begin{proof}
Note that the left hand side of the inequality can be written as
$$\Bigl|\int_Q [K(z-\xi) - K(z-z_0)] d(\nu_1-\nu_2)(\xi)\Bigl|.
$$
But, under the hypothesis on $z$, we have that $|K(z-\xi)-K(z-z_0)|\leq \frac{C|\xi-z_0|}{\lambda^2 |z-\xi|^2}$ for any $\xi\in Q$.  Plugging this estimate into the integral and taking into account the supports of $\nu_1$ and $\nu_2$,  the inequality follows.
\end{proof}

\begin{proof}[Proof of Lemma \ref{scale}]
Write $$\mathcal{A} = \{j: B^{(n)}_j \neq B^{(n)}(z^*) \text{ and } B^{(n)}_j \subset B^{(n-1)}(z^{*})\}.$$
First suppose that $\text{dist}(z, Q^{(n)}_j) \geq \tfrac{1}{4}\sqrt{r_{n-1}r_n}$ for $j\in \mathcal{A}$.  Then the hypothesis of Lemma \ref{comparison} are satisfied with $\nu_1 = \chi_{Q^{(n)}_j}\tfrac{m_2}{r_{n-1}}$, $\nu_2 
= \chi_{{B^{(n)}_j}}\mu^{(m)}$,  $R=r_{n-1}$, $r=r_n$, and $z_0=z_{Q^{(n)}_j}$.  Thus
\begin{equation}\begin{split}\label{smalleps}\Bigl|\int_{Q^{(n)}_j} K(z-\xi)&\frac{dm_2(\xi)}{r_{n-1}}-\int_{B^{(n)}_j} K(z-\xi) d\mu^{(m)}(\xi)\Bigl|\\
&\leq \int_{Q^{(n)}_j}\frac{C\sqrt{r_{n-1}r_n}}{|z-\xi|^2}\frac{dm_2(\xi)}{r_{n-1}} + \int_{B^{(n)}_j} \frac{Cr_nd\mu^{(m)}(\xi)}{|z-\xi|^2}.
\end{split}\end{equation}
Now suppose that $j\in \mathcal{A}$ and $\dist(z, Q^{(n)}_j)\leq \tfrac{1}{4}\sqrt{r_{n-1}r_n}$.  Since $\dist(z, Q^{(n)}_j) \geq \dist(z^*, Q^{(n)}_j)-\dist(z,z^*)\geq \tfrac{1}{2}\sqrt{r_{n-1}r_n}-\eps$, we must have that $\eps\geq \tfrac{1}{4}\sqrt{r_{n-1}r_n}$. But as $\dist(z,\supp(\mu^{(m)}))\geq \tfrac{\eps}{2}$, and $\mu^{(m)}(B^{(n)}_j) = r_n$, we have the following crude bound
\begin{equation}\begin{split}\label{largeps}
\Bigl|\int_{Q^{(n)}_j} &K(z-\xi)\frac{dm_2(\xi)}{r_{n-1}}-\int_{B^{(n)}_j} K(z-\xi) d\mu^{(m)}(\xi)\Bigl|\\
& \leq \frac{C}{r_{n-1}}\sqrt{m_2(Q^{(n)}_j)} + \frac{2}{\eps}\mu^{(m)}(B^{(n)}_j) \leq Cs_n.
\end{split}\end{equation}
(Here it is used that $\int_{A}|K(\xi)|dm_2(\xi) \leq C\sqrt{m_2(A)}$ for any Borel measurable set $A\subset \mathbb{C}$ of finite $m_2$-measure.)

At most $4$ of the essentially pairwise disjoint squares $Q^{(n)}_j$, $j\in \mathcal{A}$, can satisfy $\dist(z, Q^{(m)}_j) \leq \tfrac{1}{4}\sqrt{r_{n-1}r_n}$ (and it can only happen at all if $n=q$).   Therefore by summing (\ref{smalleps}) and (\ref{largeps}) over $j\in \mathcal{A}$ in the cases when $\dist(z, Q^{(n)}_j)\geq \tfrac{1}{4}\sqrt{r_{n-1}r_n}$ and $\dist(z, Q^{(n)}_j) \leq \tfrac{1}{4}\sqrt{r_{n-1}r_n}$ respectively, we see that the quantity
$$\Bigl|\int_{\bigcup_{j\in \mathcal{A}}Q^{(n)}_j} \!\!K(z-\xi) \frac{dm_2(\xi)}{r_{n-1}}-\int_{B^{(n-1)}(z^*)\backslash B^{(n)}(z^*)} \!\!K(z-\xi) d\mu^{(m)}(\xi)\Bigl|,
$$
is no greater than a constant multiple of
$$\int_{B(z,2r_{n-1})\backslash B(z,\tfrac{1}{4}\sqrt{r_nr_{n-1}})}\!\!\sqrt{\frac{r_n}{r_{n-1}}}\frac{dm_2(\xi)}{|z-\xi|^2}+  \int_{\mathbb{C}\backslash B(z,\tfrac{1}{4}\sqrt{r_nr_{n-1}})}\!\!\frac{r_nd\mu^{(m)}(\xi)}{|z-\xi|^2}+s_n.
$$
The first term here is bounded by $C\sqrt{\tfrac{r_n}{r_{n-1}}}\log\bigl(\tfrac{r_{n-1}}{r_n}\bigl)\leq Cs_n\log(\tfrac{1}{s_n})\leq C\sqrt{s_n}$.  Since $\mu^{(m)}$ is $C_0$-nice, we bound the second term by $$Cr_n\int_{\tfrac{1}{4}\sqrt{r_nr_{n-1}}}^{\infty}\frac{dr}{r^2}\leq Cr_n \frac{1}{\sqrt{r_n r_{n-1}}}\leq Cs_n.$$

We now wish to estimate $\int_{\bigcup_{j\in \mathcal{A}}Q^{(n)}_j} K(z-\xi) \frac{dm_2(\xi)}{r_{n-1}}$.   
With a slight abuse of notation, write $\widetilde{B}^{(n-1)}(z^*)=\widetilde{B}^{(n-1)}_j$ if $z^*\in B^{(n-1)}_j$.  Then
$$\Bigl|\int_{\widetilde{B}^{(n-1)}(z^*)} K(z-\xi) \frac{dm_2(\xi)}{r_{n-1}}- \int_{\bigcup_{j\in \mathcal{A}}Q^{(n)}_{j}} K(z-\xi)\frac{dm_2(\xi)}{r_{n-1}}\Bigl|
$$
is bounded by $\tfrac{C}{r_{n-1}}\bigl(m_2\bigl(\widetilde{B}^{(n-1)}(z^{*})\triangle \bigcup_{j\in \mathcal{A}}Q^{(n)}_{j}\bigl)\bigl)^{\tfrac{1}{2}}.$  By Lemma \ref{smalldiff}, this quantity is no greater than $\tfrac{C}{r_{n-1}}\sqrt{r_{n}^{1/2}r_{n-1}^{3/2}+r_nr_{n-1}} \leq C\sqrt{s_n}.$  

It remains to employ the reflectionless property (Lemma \ref{refl}).  
Since $z\in (1+\tfrac{\eps}{r_{n-1}})B^{(n-1)}(z^*)$, we use Lemma \ref{refl} to infer that
$$\Bigl|\int_{\widetilde{B}^{(n-1)}(z^*)}K(z-\xi)\frac{dm_2(\xi)}{r_{n-1}}\Bigl| = \Bigl|\int_{(1+\tfrac{\eps}{r_{n-1}})B^{(n-1)}(z^*)\backslash \widetilde{B}^{(n-1)}(z^*)}\!\!\!K(z-\xi)\frac{dm_2(\xi)}{r_{n-1}}\Bigl|.
$$
This quantity is bounded by $\tfrac{C}{r_{n-1}}\bigl(m_2((1+\tfrac{\eps}{r_{n-1}})B^{(n-1)}(z^*)\backslash \widetilde{B}^{(n-1)}(z^*))\bigl)^{\tfrac{1}{2}}\leq C\sqrt{s_n+ \tfrac{\eps}{r_{n-1}}}$.  The lemma follows.
\end{proof}

With Lemma \ref{scale} in hand, we may complete the proof of Proposition \ref{maxest}.  First write
\begin{equation}\begin{split}\label{splitup}\int_{\mathbb{C}}K(z-\xi) d\mu^{(m)}(\xi)& = \int_{B^{(q)}(z^*)}K(z-\xi) d\mu^{(m)}(\xi) \\
&+\sum_{n=1}^q\int_{B^{(n-1)}(z^*)\backslash B^{(n)}(z^*)}K(z-\xi) d\mu^{(m)}(\xi).
\end{split}\end{equation}
Next note that that $\tfrac{\eps}{r_{n-1}} \leq 1$ if $n= q$, and $\sqrt{\tfrac{\eps}{r_{n-1}}}\leq s_n$ for $1\leq n<q$.  As $\sum_{n\geq 1} \sqrt{s_n} <\infty,$
it follows from Lemma \ref{scale} that the sum appearing in the right hand side of (\ref{splitup}) is bounded in absolute value independently of $q$, $m$ and $\eps$.  The remaining term on the right hand side of (\ref{splitup}) has already been shown to be bounded in absolute value, see (\ref{qest}).

\section{$T_{\mu}(1)$ fails to exist in the sense of principal value $\mu$-almost everywhere}

We now turn to consider the operator in the sense of principal value.  The primary part of the argument will be the following lemma.

\begin{lem}\label{misbalance}  Provided that $n$ is sufficiently large, there exists a constant $c_0>0$ such that for any disc $B^{(n)}_j$, and $z\in \mathbb{C}$ satisfying \newline$\dist(z, \partial B^{(n)}_j)\leq c_0r_n$,
$$\Bigl|\int_{A(z, r_n)} K(z-\xi) d\mu(\xi)\Bigl|\geq c_0.
$$
\end{lem}

Before proving the lemma, we deduce from it that $T_{\mu}(1)$ fails to exist in the sense of principal value for $\mu$-almost every $z\in \mathbb{C}$.  To this end, we set $F=\{z\in E: z \in (1-c_0)B^{(n)}(z) \text{ for all but finitely many }n\}$.  It suffices to show that $\mu(F)=0$.

First note that, with  $F_n=\{z\in E: z \in (1-c_0)B^{(m)}(z) \text{ for all }m\geq n\}$, we have $F\subset \bigcup_{n\geq 0 } F_{n}$, so it suffices to show that $\mu(F_n)=0$ for all $n$.

To do this, note that for each $m\geq 0$, at most $(1-c_0)\tfrac{r_m}{r_{m+1}} + C\sqrt{\tfrac{r_m}{r_{m+1}}}$ squares $Q^{(m+1)}_{\ell}$ can intersect $(1-c_0)B^{(m)}_j$.  
Thus
\begin{equation}\begin{split}\nonumber\mu\Bigl(\bigcup_{\ell}\Bigl\{B^{(m+1)}_{\ell}\!: & B^{(m+1)}_{\ell}\!\cap \!(1-c_0)B^{(m)}_j \neq \varnothing\Bigl\}\Bigl)\leq (1-c_0)r_m + C\sqrt{\frac{r_{m}}{r_{m+1}}}r_{m+1}\\
&= (1-c_0)\mu(B^{(m)}_j)+ Cs_{m+1}r_m\leq \Bigl(1-\frac{c_0}{2}\Bigl)\mu(B^{(m)}_j),
\end{split}\end{equation}
where the last inequality holds provided that $m$ is sufficiently large.  But then, as long as $n$ is large enough, this inequality may be iterated to yield
$$\mu\bigl(\bigl\{z\in E: z\in (1-c_0)B^{(n+k)}(z) \text{ for }k=1,\dots,m\bigl\}\bigl) \leq (1-\tfrac{c_0}{2})^m.$$ Hence $\mu(F_n)=0$.

In preparation for proving Lemma \ref{misbalance}, we make the following claim.

\begin{cla}\label{anncomp}   Let $n\in \mathbb{Z}_+$.  For any disc $B^{(n)}_j$, and $z\in \mathbb{C}$, we have
$$\Bigl|\int_{A(z,r_n)\cap B^{(n)}_j}K(z-\xi) d\bigl(\mu - \frac{m_2}{r_{n}}\bigl)(\xi)\Bigl|\leq Cs_{n+1}.
$$
\end{cla}

\begin{proof} To derive this claim, first suppose that a square $Q^{(n+1)}_{\ell}\subset A(z, r_n)$.  Then from a crude application of Lemma \ref{comparison} (see (\ref{smalleps})), we infer that
$$\Bigl|\int_{Q^{(n+1)}_{\ell}} K(z-\xi) d(\mu - \frac{m_2}{r_{n}})(\xi)\Bigl|\leq C \frac{\sqrt{r_nr_{n+1}}}{r_n^2} r_{n+1}\leq C\Bigl(\frac{r_{n+1}}{r_n}\Bigl)^{\tfrac{3}{2}}.
$$
If it instead holds that $Q^{(n+1)}_{\ell}\cap \partial A(z, r_n) \neq \varnothing$, then we have the blunt estimate
$$\Bigl|\int_{Q^{(n+1)}_{\ell}\cap A(z, r_n)} K(z-\xi) d(\mu - \frac{m_2}{r_{n}})(\xi)\Bigl|\leq \frac{2}{r_n}\Bigl[\mu(Q^{(n+1)}_{\ell})+\frac{m_2(Q^{(n+1)}_{\ell})}{r_n}\Bigl],$$
which is bounded by $\tfrac{Cr_{n+1}}{r_n}$.  There are most $\frac{r_n}{r_{n+1}}$ squares $Q^{(n+1)}_{\ell}$ contained in $A(z,r_n)$, and no more than $C\sqrt{\tfrac{r_n}{r_{n+1}}}$ squares $Q^{(n+1)}_{\ell}$ can intersect the boundary of $A(z, r_n)$.

On the other hand, the set $\widetilde{A}$ consisting of the points in $A(z, r_{n})\cap B^{(n)}_j$ not covered by any square $Q^{(n+1)}_{\ell}$ has $m_2$-measure no greater than $Cr_{n+1}^{1/2}r_{n}^{3/2}$ (see  Lemma \ref{smalldiff}).  Thus $\int_{\widetilde{A}}|K(z-\xi)| \tfrac{dm_2(\xi)}{r_{n}}\leq\tfrac{2m_2(\widetilde{A})}{r_n^2} \leq Cs_{n+1}$.

Bringing these estimates together establishes Claim \ref{anncomp}. \end{proof}

Let us now complete the proof of Lemma \ref{misbalance}

\begin{proof}[Proof of Lemma \ref{misbalance}]
Note that $\int_{A(z, r_n)\cap B^{(n)}_j} K(z-\xi) \tfrac{dm_2(\xi)}{r_{n}}$ is a Lipschitz continuous function in $\mathbb{C}$, with Lipschitz norm at most $\tfrac{C}{r_n}$.  Thus, we infer from Lemma \ref{lebmisbal} that there is a constant $c_0>0$ such that
$$\Bigl|\int_{A(z, r_n)\cap B^{(n)}_j} K(z-\xi) \frac{dm_2(\xi)}{r_n}\Bigl|\geq \frac{\tilde{c}}{2},
$$
whenever $\dist(z, \partial B^{(n)}_j)\leq c_0 r_n$.  But now we apply Claim \ref{anncomp} to deduce that for all such $z$,
$\bigl|\int_{A(z,r_n)} K(z-\xi) d\mu(\xi)\bigl|\geq \tfrac{\tilde{c}}{2}-Cs_{n+1}$  (the only part of the support of $\mu$ that $A(z,r_n)$ intersects is contained in $B^{(n)}_j$).  The right hand side here is at least $\tfrac{\tilde{c}}{4}$
for all sufficiently large $n$.
\end{proof}

\section{The set $E$ is purely unrectifiable}

We now show that $E$ is purely unrectifiable, that is, $\mathcal{H}^1(E\cap \Gamma)=0$ for any rectifiable curve $\Gamma$.  The proof that follows is a simple special case of the well known fact that any set with zero lower $\mathcal{H}^1$-density is unrectifiable (one can in fact say much more, see for instance \cite{Mat95}).

First notice that for each $z\in \mathbb{C}$ and $n\geq 1$, $B(z, \tfrac{1}{4}\sqrt{r_nr_{n-1}})$ can intersect at most one of the discs $B^{(n)}_j$. Hence $$\mathcal{H}^1(E\cap B(z, \tfrac{1}{4}\sqrt{r_nr_{n-1}})) \leq 2r_n.$$

A rectifiable curve $\Gamma$ can be covered by discs $B(z_j, \tfrac{1}{4}\sqrt{r_nr_{n-1}})$, $j=1,\dots, N$, the sum of whose radii is at most $\ell(\Gamma)$.

Thus
$\mathcal{H}^1(E\cap \Gamma) \leq \sum_{j=1}^N \mathcal{H}^1(E\cap B(z_j, \tfrac{1}{4}\sqrt{r_nr_{n-1}}))\leq 2\sum_{j=1}^N r_n.$
But $ \sum_{j=1}^N \frac{1}{4}\sqrt{r_nr_{n-1}}\leq \ell(\Gamma)$, and so $ \mathcal{H}^1(\Gamma\cap E)\leq 8\sqrt{\frac{r_n}{r_{n-1}}}\ell(\Gamma)$, which tends to zero as $n\rightarrow \infty$  (the sequence $\sqrt{s_n}$ is summable).

 \end{document}